\documentclass[10pt]{amsart}
\usepackage{geometry} 
\usepackage{color}

\newtheorem{theorem}{Theorem}[section]
\newtheorem{corollary}[theorem]{Corollary}

\newtheorem{lemma}[theorem]{Lemma} 
\newtheorem{proposition}[theorem]{Proposition}
\makeatletter
\numberwithin{equation}{section}
\makeatother

\def\R{\mathbf{R}}
\def\M{\mathbf{M}}

\def\F{\mathbf F}
\def\H{\mathcal H}
\def\bh{\mathbf h}

\def\bbB{\mathbb B}
\def\H{\mathcal H}
\def\G{\mathbf G}
\def\I{\mathcal I}

\def\cC{\mathbb C}

\def\id{\mathbf{1}}
\def\bcero{\mathbf{0}}
\def\bx{\mathbf{x}}

\title[Mathematical programming]{A dynamical-system perspective on mathematical programming}
\author{Pablo Pedregal}
\thanks{
INEI. Universidad de Castilla La Mancha.
Campus de Ciudad Real (Spain). Research supported in part by
MTM2013-47053-P of the MINECO (Spain).
 e-mail:{\tt pablo.pedregal@uclm.es
}}

\begin{document}

\maketitle
    \begin{abstract}
We explore how to build a vector field from the various functions involved in a given mathematical program, and show that equilibria of the corresponding dynamical system are precisely the solutions of the underlying optimality conditions for the original optimization problem. The general situation in which explicit inequality constraints are present is especially interesting as the vector field has to be discontinuous, and so one is led to consider discontinuous dynamical systems and their equilibria. 
    \end{abstract}


\section{Introduction}
It is well-known that 
the mathematical program
\begin{equation}\label{minimum}
\hbox{Minimize in }\mathbf{x}\in\R^N:\quad f(\mathbf{x})
\end{equation}
for a given smooth function $f$, 
is intimately connected to the underlying (gradient) dynamical system
\begin{equation}\label{flujo}
\mathbf{x}'=-\nabla f(\mathbf{x}).
\end{equation}
Roughly speaking, isolated asymptotically stable equilibria for this system correspond to locally stable solutions (local minima) of the optimization problem above. 
In fact, it is a rather delicate issue to distinguish between local minimality and stability. See \cite{absilkurdyka} for a complete and careful discussion about the relationship between these concepts. To avoid here the discussion on these fine points, which is not part of our objective, we will assume that equilibria of all the dynamical systems involved are isolated. Under this main assumption, it is shown in \cite{absilkurdyka} that local solutions of (\ref{minimum}) (local minima of $f$) exactly correspond to asymptotically stable equilibria of (\ref{flujo}). 

A complementary viewpoint focuses on a vector field $\Psi(\bx)$, and its isolated equilibria. One such equilibrium point $\bx_0$ is locally asymptotically stable if there is a (Lyapunov) function $f$ for which $\bx_0$ is an isolated local minima and $\nabla f(\bx)\cdot\Psi(\bx)<0$, at least in a vicinity of $\bx_0$, except at $\bx_0$ itself. Whenever this is so, the integral curves of $\Psi$ will take us to the local minimum $\bx_0$ of $f$. Obviously, the field $\Psi=-\nabla f$ is one such possibility. 

If we can (and we indeed do) count on efficient numerical algorithms to approximate integral curves for (\ref{flujo}) or of those corresponding to $\Psi$, for which $f$ is a Lyapunov function in a neighborhood of a local minimum, then we can use those to find good approximations for  local solutions of the optimization problem (\ref{minimum}). This is essentially the steepest descent strategy and important variants. We therefore can count on two ways to find approximations for the asymptotically stable critical points of $f$ (local minima):
\begin{enumerate}
\item approximate directly the system of equations $\nabla f(\mathbf{x})=\bcero$, and discern which of those correspond to local minima;
\item use the flow of $-\nabla f$, or of $\Psi$ with $\Psi\cdot\nabla f<0$, to get close to such points.
\end{enumerate}
Our intention is to explore this perspective for a general mathematical program involving both equality and inequality constraints
\begin{equation}\label{probgen}
\hbox{Minimize in }\mathbf{x}\in\R^N:\quad f(\mathbf{x})\quad \hbox{subject to}\quad \mathbf{g}(\mathbf{x})\le\bcero, \mathbf{h}(\mathbf{x})=\bcero,
\end{equation}
where $f:\R^N\to\R$, $\mathbf{g}:\R^N\to\R^m$, $\mathbf{h}:\R^N\to\R^n$, are continuously differentiable. The first strategy above (the one dealing directly with optimality conditions) has been traditionally followed in finding good approximations of locally-stable solutions. This amounts to dealing directly with the Karush-Kuhn-Tucker (KKT) optimality conditions when constraints are to be respected. These are
\begin{equation}\label{kkt}
\nabla f(\mathbf{x})+\mu\nabla \mathbf{g}(\mathbf{x})+\lambda\nabla \mathbf{h}(\mathbf{x})=\bcero,\quad \mu\ge\bcero, \mu \mathbf{g}(\mathbf{x})=0.
\end{equation}
Bear in mind that $\nabla f$ is a $1\times N$-matrix (a vector in $\R^N$), while $\nabla \mathbf{g}$ and $\nabla \mathbf{h}$ are $m\times N$-, and $n\times N$ matrices, respectively; $\mu\in\R^m$, $\lambda\in\R^n$. In addition, the point $\mathbf{x}$ must, of course, be feasible $\mathbf{g}(\mathbf{x})\le\bcero$, $\mathbf{h}(\mathbf{x})=\bcero$. 

Though one can 
still take advantage of those descent strategies by incorporating penalizations designed in various ways to account for those constraints (see, for instance, \cite{hiriarturruty}, \cite{nocedal}), our perspective here is a bit different from these penalization techniques, and stresses the dynamical viewpoint of the second way indicated above:  
\begin{quotation}
Given a mathematical programming problem, design a suitable vector field $\Psi(\bx)$ whose asymptotically stable equilibria are precisely the local solutions of \eqref{probgen}, i.e., the solutions of the underlying Karush-Kuhn-Tucker system \eqref{kkt}.
\end{quotation}
Whenever this objective is accomplished, integral curves for $\Psi$ will take us to local minima of \eqref{probgen}. 

If constraints are an important part of our optimization problem,
it is not clear, to begin with, if solutions of the optimality system might exactly correspond to equilibria for a certain dynamical system. As we have explained, this is the main goal of this contribution: to show how to build such a vector field $\Psi$, and prove that identification between  equilibria, and solutions of the KKT optimality system. This strategy is quite natural, and indeed have been pursued before in various works (see for instance \cite{jongenstein}, \cite{pappalardopassacantando}). However, the dynamical systems (and the dynamics itself) is built in such a way that the feasible set and/or  constraints are constantly being monitored in the dynamical process. From this perspective, we would like to adopt a point of view mainly concerned with global dynamics, so that once the vector field has been determined through the various ingredients of the underlying mathematical program, one forgets altogether about feasibility or constraints, and focuses on the dynamics. 

To make the issue clear, and motivate our objective, 
let us focus on the following simple mathematical program
\begin{equation}\label{probfin}
\hbox{Minimize in }\mathbf{x}\in\R^N:\quad f(\mathbf{x})\quad \hbox{subject to}\quad h(\mathbf{x})=0,
\end{equation}
where both $f, h$ are smooth real functions. We assume that the hypersurface determined by the equation $h(\bx)=0$ is regular at every point, so that $\nabla h$ never vanishes over $h=0$.
Recall that KKT optimality conditions for such a problem read
\begin{equation}\label{optimm}
\nabla f(\mathbf{x})+\lambda\nabla h(\mathbf{x})=\bcero,\quad 
h(\mathbf{x})=0,
\end{equation}
for a multiplier $\lambda$. 
This situation is intimately related to the Lyapunov-function idea but for a function which is not positive definite. If $\Psi$ is a smooth vector field such that $h\, \Psi\cdot\nabla h<0$ whenever $h\neq0$, but the zero set $h=0$ is not a single vector, then the dynamics of $\Psi$ preserves the manifold $h=0$ ($\Psi\cdot\nabla h=0$ if $h=0$), and a lot of things may happen there. Our inspiring result is the following.

\begin{lemma}\label{primero}
Suppose $\overline\bx$ is an isolated local solution of the optimization problem \eqref{probfin} and of \eqref{optimm} 
under the indicated hypotheses on $f$ and $h$. Let $\Psi:\R^N\to\R^N$ be a smooth vector field such that:
\begin{enumerate}
\item $h(\Psi\cdot\nabla h)<0$ in a neighbourhood of $h=0$ but $h\neq0$;
\item $\Psi\cdot\nabla f\le0$ if $h=0$;
\item $\Psi\cdot\nabla f=0$ and $h=0$ imply that $\nabla f$ and $\nabla h$ are parallel.
\end{enumerate}
Then integral curves of the dynamical system $\bx'=\Psi(\bx)$ starting sufficiently close to $\overline\bx$ will converge to $\overline\bx$, namely,
\begin{equation}\label{limit}
\lim_{t\to\infty}\bx(t)=\overline\bx
\end{equation}
provided $\bx(0)$ is already sufficiently close to $\overline\bx$. 
\end{lemma}
\begin{proof}
We argue by contradiction. Suppose \eqref{limit} does not hold always even if starting points $\bx(0)$ are taken sufficiently close to $\overline\bx$. This means that we can find a sequence of vectors $\bx_j(0)\to\overline\bx$, and yet the corresponding integral curves
$$
\cC_j=\{\bx_j(t): t\ge0\}
$$
of  $\bx'=\Psi(\bx)$ are such that there is some sufficiently small positive $r_0$ with
$$
\cC_j\cap\bbB(\overline\bx, r_0)
$$
 not the full $\cC_j$ for $j$ large. Hence, there is some $t_j>0$, so that
 $$
 \bx_j(t_j)\in\partial\bbB(\overline\bx, r_0).
 $$
Because of the  first  condition on $\Psi$ related to $\nabla h$, 
it is clear that
\begin{equation}\label{desigualdad}
\frac{d}{dt}\left[\frac12 h(\bx(t))^2\right]<0
\end{equation}
in the vicinity mentioned in the statement, for every integral curve $\bx(t)$. 
Consequently, the sign of $h$ is preserved in all of the $\cC_j$'s, i.e., $\cC_j$ cannot cross through the manifold $h=0$. By continuity, and bearing in mind again that the size of $h$ decreases steadily over integral curves of $\Psi$ according to \eqref{desigualdad}, there must be a limit integral curve $\bx_\infty(t)$, fully contained in the manifold $h=0$, and
\begin{equation}\label{contradiction}
\bx_\infty(0)=\overline\bx,\quad \overline{\{\bx_\infty(t): t\ge0\}}\cap\bbB(\overline\bx, r_0)\neq\emptyset.
\end{equation}
On the other hand, 
by the other two properties of $\Psi$ related to $\nabla f$ in the statement, over the manifold $h=0$, $f$ is a Lyapunov function for the dynamical system, and there is a unique minimum, $\overline\bx$, of $f$ over $h=0$ in the ball $\bbB(\overline\bx, r_0)$ if $r_0$ is sufficiently small. This means that integral curves of $\bx'=\Psi(\bx)$ on $h=0$ starting sufficiently close to $\overline\bx$ must converge to $\overline\bx$, against \eqref{contradiction}. This contradiction proves that \eqref{limit} is correct as long as $\bx(0)$ is sufficiently close to $\overline\bx$. In particular, we must have $\Psi(\overline\bx)=\bcero$. 
\end{proof}
This lemma is telling us what properties we need on a field $\Psi$ to guarantee that integral curves will take us to the (local) solutions of a minimization problem under constraints like \eqref{probfin}. It is not difficult to find such a $\Psi$ based on the two gradients $\nabla f$ and $\nabla h$. In the statement that follows, $\id$ designates the identity matrix of size $N\times N$, and $\mathbf{a}\otimes \mathbf{b}=\mathbf{a}^T\mathbf{b}$ is the rank-one matrix which is the tensor product of two vectors in $\R^N$. 

\begin{proposition}\label{igualdad}
Assume $\nabla h$ never vanishes over the manifold $\{h=0\}$, and consider the vector field 
\begin{equation}\label{campoo}
\Psi(\mathbf{x})=-h(\mathbf{x})\nabla h(\mathbf{x})-\nabla f(\mathbf{x})\left(\id-\frac{\nabla h(\mathbf{x})}{|\nabla h(\mathbf{x})|}\otimes\frac{\nabla h(\mathbf{x})}{|\nabla h(\mathbf{x})|}\right),
\end{equation}
which is smooth in a vicinity of the manifold $\{h=0\}$. 
Suppose that local solutions of \eqref{probfin} are isolated. 
Then asymptotically stable equilibria of $\Psi$ in that neighbourhood of $\{h=0\}$ are precisely the local solutions (minima) of the mathematical programming problem (\ref{probfin}).
\end{proposition}
Note that the second term in $\Psi$ is the orthogonal projection of $\nabla f(\bx)$ onto the orthogonal complement of the subspace spanned by $\{\nabla h(\mathbf{x})\}$. If we identify row and column vectors in the canonical way, $\Psi$ in \eqref{campoo} can also be written
$$
\Psi(\mathbf{x})=-h(\mathbf{x})\nabla h(\mathbf{x})^T-\left(\id-\frac{\nabla h(\mathbf{x})}{|\nabla h(\mathbf{x})|}\otimes\frac{\nabla h(\mathbf{x})}{|\nabla h(\mathbf{x})|}\right)\nabla f(\mathbf{x})^T,
$$
or even
$$
\Psi(\mathbf{x})=-h(\mathbf{x})\nabla h(\mathbf{x})-\left(\id-\frac{\nabla h(\mathbf{x})}{|\nabla h(\mathbf{x})|}\otimes\frac{\nabla h(\mathbf{x})}{|\nabla h(\mathbf{x})|}\right)\nabla f(\mathbf{x})^T.
$$
\begin{proof}
 After the previous lemma, all that is required to check is that $\Psi$, defined by \eqref{campoo}, complies with all those requirements in the statement of the lemma. It is elementary to check that it is so. 
\end{proof}

It is clear how the vector field $\Psi$ has been tailored precisely to have as equilibria the (local) solutions of the mathematical program. If we knew the precise value of the multiplier $\lambda$ before hand, then obviously we could examine the dynamical system
$$
\bx'(t)=-\nabla f(\bx(t))-\lambda \nabla h(\bx(t))
$$
near the manifold $\{h=0\}$. 
The whole point is to design a vector field taking us to the solutions of programming problems without knowing multipliers. 

As indicated earlier, the practical implication of this approach can hardly be hidden: by approximating integral curves for $\Psi$ through typical but refined numerical procedures, we can find approximations of the solutions of mathematical programs. One of the most appealing advantages of this approach is that one can start the dynamical process with arbitrary initializations without taking any care about whether they are feasible or not. 
Under further constraints, global stability may be established. 

\begin{corollary}
Suppose $f:\R^N\to\R$ is coercive in the sense
$$
\lim_{|\mathbf{x}|\to\infty}f(\mathbf{x})=+\infty,
$$
and $\nabla h$ never vanishes. If there is a unique solution for the corresponding KKT optimality conditions, then every integral curve of the associated vector field $\Psi$ converges to such a unique solution.
\end{corollary}
Explicit assumptions ensuring the uniqueness in this statement amount to asking for the strict convexity of $f$ and the linearity (affinity) of $h$ as in classic results of sufficiency of optimality conditions (\cite{nocedal}). 

The main concern in this contribution is to design a field $\Psi$ for a typical non-linear mathematical program in full generality
\begin{equation}\label{probgen0}
\hbox{Minimize in }\mathbf{x}\in\R^N:\quad f(\mathbf{x})\quad \hbox{subject to}\quad \mathbf{g}(\mathbf{x})\le\bcero, \mathbf{h}(\mathbf{x})=\bcero,
\end{equation}
where $f:\R^N\to\R$, $\mathbf{g}:\R^N\to\R^m$, $\mathbf{h}:\R^N\to\R^n$, are smooth, and comply with typical constraint qualifications that will be indicated appropriately. In particular, we would like to understand the different role played by inequality and equality constraints. Note that the vector field in (\ref{campoo}) is continuous provided that $\nabla h$ never vanishes. This is the situation for problems with just equality constraints, as we will briefly sketch in Section 2. However, when, in addition to equality, there are also inequality constraints to be respected, the corresponding vector field cannot be continuous (though it is piecewise continuous). Associated with (\ref{probgen0}), our proposal consists in considering the vector field
\begin{equation}\label{campo}
\Psi(\mathbf{x})=\begin{cases}-\mathbf{h}(\mathbf{x})\nabla \mathbf{h}(\mathbf{x})-\left(\id-\mathbf{H}(\mathbf{x})\right)\G(\mathbf{x})^T,& G(\mathbf{x})>0,\\
-\mathbf{h}(\mathbf{x})\nabla \mathbf{h}(\mathbf{x})-\left(\id-\mathbf{H}(\mathbf{x})\right)\nabla f(\mathbf{x})^T,& G(\mathbf{x})\le0,
\end{cases}
\end{equation}
where
\begin{gather}
G(\mathbf{x})=\max{g_i(\mathbf{x})},\quad \I(\mathbf{x})=\{i: g_i(\mathbf{x})=G(\mathbf{x})\},\quad \G(\mathbf{x})=\sum_{i\in\I(\mathbf{x})}\nabla g_i(\mathbf{x}),\nonumber\\
 \mathbf{H}(\mathbf{x})=\nabla \mathbf{h}(\mathbf{x})^T(\nabla \mathbf{h}(\mathbf{x})\nabla \mathbf{h}(\mathbf{x})^T)^{-1}\nabla \mathbf{h}(\mathbf{x}).\nonumber
\end{gather}
This vector field is clearly not continuous through $G=0$. This fact leads us to the theory of discontinuous dynamical systems. 
Despite the fact that this is a rather specialized area which is not a typical part of dynamical system theory, the subject is rather well understood, at least as long as our needs are concerned here. The standard reference for this subject is \cite{filippov}, although there are very nice recent accounts (\cite{cortes}). 

Our main result follows. We further put
$$
\H(\mathbf{x})=\hbox{subspace spanned by }\{\nabla \mathbf{h}(\mathbf{x})\}.
$$
\begin{theorem}\label{variasdes}
Let $f:\R^N\to\R$, $\mathbf{h}:\R^N\to\R^n$, $\mathbf{g}:\R^N\to\R^m$ be continuously differentiable maps such that
\begin{itemize}
\item $\nabla \mathbf{h}(\mathbf{x})\nabla \mathbf{h}(\mathbf{x})^T$ is an invertible $n\times n$-matrix for all $\mathbf{x}$ in a neighbourhood of the set $\bh=\bcero$;
\item for each $\mathbf{x}$ where $G(\mathbf{x})=0$, the intersection of the subspaces $\H(\mathbf{x})$ and the one spanned by $\{\nabla g_i(\mathbf{x}): i\in\I(\mathbf{x})\}$ is the trivial subspace;
\item the null vector is never in the convex hull of the set $\{\nabla g_i(\mathbf{x}): i\in\I(\mathbf{x})\}$ when $G(\mathbf{x})=0$;
\item solutions of the corresponding Karush-Kuhn-Tucker optimality conditions
\begin{equation}\label{KKTd}
\nabla f(\mathbf{x})+\mu\nabla \mathbf{g}(\mathbf{x})+\lambda\nabla \mathbf{h}(\mathbf{x})=\bcero,\quad \mu\ge\bcero, \mu \mathbf{g}(\mathbf{x})=0,
\end{equation}
complying with $\mathbf{g}(\mathbf{x})\le\bcero$, $\mathbf{h}(\mathbf{x})=\bcero$, are isolated.
\end{itemize}
Then,  asymptotically stable equilibria for $\Psi$ in \eqref{campo} (in the generalized sense of Filippov) are exactly the feasible points $\mathbf{x}$, $\mathbf{g}(\mathbf{x})\le\bcero$, $\mathbf{h}(\mathbf{x})=\bcero$, which are solutions of (\ref{KKTd}), and correspond to a local optimal solution (minima) of  (\ref{probgen0}). 
\end{theorem}

After dealing with the general situation for mathematical programs with no inequality constraints in Section 2, we concentrate, in Section 3, in a simple problem with just one inequality constraint to more clearly see the need for understanding discontinuous dynamical systems. It can also serve as a first explicit example of a discontinuous dynamical system whose equilibria correspond to (local) solutions of mathematical programs. 
We also include in this section a discussion on the basic facts about Filippov's extension of  such fields. This material is taken essentially from \cite{filippov}. For such systems, there is a number of generalized notions of solutions across an interface of discontinuity. The one relevant for our purposes is the one introduced by Filippov leading to the notion of sliding modes. It is apparently also the one suitable for mechanics (\cite{cortes}), but it  adapts perfectly to our objective as well. After that, Section 4 contains a preliminary discussion when there is just one inequality constraint present, but possibly many equality restrictions. It serves us as a preliminary step towards the final result. Section 5 extends those ideas to the full, general mathematical program. Finally, Section 6 focuses briefly on the 
particular case of linear programming. 

Though the interest on discontinuous dynamical system is quite recent, there is a growing body of literature about this topic. 
There is however not many references on the specific topic of this work, as the perspective is, to the best of our knowledge, new. Concerning equilibria of discontinuous dynamical systems, see \cite{llibremereu}, \cite{llibreponce}. In \cite{gughai}, a close analysis is performed about some  situations of discontinuous dynamical systems through transversal interfaces. 
For the numerical approximation of discontinuous systems, one can look at the nice review \cite{calmonran}. \cite{smirnov} is a standard reference for differential inclusions which is also an intimately related area as we will recall. 

There are several important issues to be explored once the way to build the underlying vector field $\Psi$ has been well established, and its main property (Theorem \ref{variasdes}) proved. We name three such future directions which might be relevant to assess the possible impact of this perspective on both areas, discontinuous dynamical systems and mathematical programming:
\begin{enumerate}
\item study of equilibra of discontinuous vector fields coming from explicit, well-chosen mathematical programs;
\item global stability for the discontinuous $\Psi$ corresponding to a general mathematical program for which sufficiency of optimality conditions guarantee a unique global minimum;
\item test the numerical implementation to approximate solutions of mathematical programs by following the integral curves of the underlying field $\Psi$.
\end{enumerate}

\section{The equality situation}
Let us first deal with the situation of a mathematical program in which no inequality is involved, but we could have many equality constraints
\begin{equation}\label{probeq}
\hbox{Minimize in }\mathbf{x}\in\R^N:\quad f(\mathbf{x})\quad \hbox{subject to}\quad \mathbf{h}(\mathbf{x})=\bcero,
\end{equation}
where $f:\R^N\to\R$, $\mathbf{h}:\R^N\to\R^n$ ($n<N$), are smooth. Note that the Jacobian matrices $\nabla f$, $\nabla \mathbf{h}$ are $1\times N$-, and $n\times N$-matrices, respectively. We assume that the manifold determined by the equations $\bh=\bcero$ is a non-empty, smooth manifold. This means that in a suitable neighbourhood of such set
\begin{equation}\label{compatibilidad}
h_i h_j \nabla h_i\cdot\nabla h_j\ge0,\quad \bh=(h_i),
\end{equation}
for all $i, j$. This condition establishes that the gradients of the square of the components 
$$
\nabla \left[\frac12 h_i(\bx)^2\right]
$$
cannot have different sign in such a  neighbourhood, because if different  signs of two such gradients persist arbitrarily close to the constraint set, that constraint set could not be the zero set of $\bh$. 

The generalization of the results in the Introduction to this more general setting is pretty straightforward. Consider the vector field $\Psi:\R^N\to\R^N$ given by
\begin{equation}\label{ge}
\Psi(\mathbf{x})=-\mathbf{h}(\mathbf{x})\nabla \mathbf{h}(\mathbf{x})-\left(\id-\mathbf{H}(\mathbf{x})\right)\nabla f(\mathbf{x})^T, \quad \mathbf{H}(\mathbf{x})=\nabla \mathbf{h}(\mathbf{x})^T(\nabla \mathbf{h}(\mathbf{x})\nabla \mathbf{h}(\mathbf{x})^T)^{-1}\nabla \mathbf{h}(\mathbf{x}).
\end{equation}
For $\Psi$ to be well-defined and  continuous, we need $f$ and $\mathbf{h}$ to be continuously differentiable, and the constraint qualification that the $n\times n$-matrix $\nabla \mathbf{h}(\mathbf{x})\nabla \mathbf{h}(\mathbf{x})^T$ be invertible for all $\mathbf{x}\in\R^N$, or at least when $\bh=\bcero$. As above, the second term in $\Psi$ accounts for the projection, onto the orthogonal complement of the subspace generated by the gradients of the constraints, of $\nabla f$. For future reference, we will put
$$
\mathbf{H}(\mathbf{x})\equiv\nabla \mathbf{h}(\mathbf{x})^T(\nabla \mathbf{h}(\mathbf{x})\nabla \mathbf{h}(\mathbf{x})^T)^{-1}\nabla \mathbf{h}(\mathbf{x}),\quad \H(\mathbf{x})=\hbox{subspace spanned by }\{\nabla \mathbf{h}(\mathbf{x})\}.
$$

For the sake of completeness, we restate a proposition and corollary as in the Introduction. 

\begin{proposition}\label{igualdadd}
Suppose that the functions involved in (\ref{probeq}) are continuously differentiable, and the $n\times n$-matrix $\nabla \mathbf{h}(\mathbf{x})\nabla \mathbf{h}(\mathbf{x})^T$ is always invertible (at least in a neighbourhood of $\{\bh=\bcero\}$), so that the associated vector field $\Psi$ is well-defined and continuous. Suppose that the solutions of the underlying KKT system 
\begin{equation}\label{kakate}
\nabla f(\mathbf{x})+\lambda\nabla \mathbf{h}(\mathbf{x})=\bcero,\quad \mathbf{h}(\mathbf{x})=\bcero,
\end{equation}
are isolated. 

Then equilibria of $\Psi$ are exactly the solutions of this system (and thus isolated), and those that are asymptotically stable correspond to local optimal solutions (minima) of (\ref{probeq}). 
\end{proposition}

The proof of such a result is based on the main properties of the vector field $\Psi$, similar to those stated in Lemma \ref{primero}, namely,
\begin{enumerate}
\item $h_i(\Psi\cdot\nabla h_i)<0$ in a neighbourhood of $h_i=0$ ($h_i\neq0$) for all $i$;
\item $\bh=\bcero$ implies $\Psi\cdot\nabla \bh=\bcero$;
\item $\Psi\cdot\nabla f\le0$ if $\bh=\bcero$;
\item $\Psi\cdot\nabla f=0$ and $\bh=\bcero$ imply that $\nabla f\in \H$.
\end{enumerate}
These properties are elementary to check for $\Psi$ given in \eqref{ge}. Recall \eqref{compatibilidad}. 

For a global result, we need to have more strict hypotheses on the structure of the mathematical progam. 
\begin{corollary}
Suppose, in addition to the required smoothness, that  $f:\R^N\to\R$ is coercive in the sense
$$
\lim_{|\mathbf{x}|\to\infty}f(\mathbf{x})=+\infty,
$$
and $\nabla \mathbf{h}(\mathbf{x})\nabla \mathbf{h}(\mathbf{x})^T$ is always non-singular. 
If there is a unique solution for (\ref{kakate}), then every integral curve of the associated vector field $\Psi$ converges to such a unique solution.
\end{corollary}

\section{A discontinuous vector field}
To better understand the source of the discontinuity of the vector field, let us explore a mathematical program with a single inequality constraint and no equality restriction
\begin{equation}\label{probsen}
\hbox{Minimize in }\bx\in\R^N:\quad f(\bx)\quad\hbox{ subject to }\quad g(\bx)\le0.
\end{equation}
It is tempting to play with formula \eqref{campoo}, and try to modify it to accommodate, in a continuous way, the constraint $g\le0$ in the inequality form. After all, the restriction set can be equivalently written $g^+(\bx)=0$ where $g^+$ is the positive part of $g$. However, the behavior of the field $\Psi$ has to be drastically (discontinuously) different when $g>0$, pursuing the effect of a field like \eqref{campoo}, from points where $g\le0$, in which case $f$ has to be a Lyapunov function for $\Psi$. 
Our proposal in this situation is 
\begin{equation}\label{propuesta}
\Psi(\bx)=\begin{cases}-\nabla g(\bx),& g(\bx)>0,\\-\nabla f(\bx),& g(\bx)\le0.\end{cases}
\end{equation}
The main issue is how to detect equilibria $\bx$ in the interface where $g(\bx)=0$, since demanding $\Psi(\bx)=\bcero$ does not make much sense for such points. Notice that, even if $\Psi$ has been defined on the interface as $-\nabla f$, these values are absolutely irrelevant: what matters is the behaviour at both sides of the interface. The theory of discontinuous dynamical systems that we need is Filippov's. 

Recall that KKT optimality conditions for \eqref{probsen} read
\begin{equation}\label{optsen}
\nabla f(\bx)+\mu\nabla g(\bx)=\bcero,\quad \mu g(\bx)=0, \quad\mu\ge0, g(\bx)\le0.
\end{equation}

\subsection{Filippov solutions for discontinuous vector fields}
One important basic reference for solutions of discontinuous dynamical systems is \cite{filippov}. This is especially so because we are interested precisely in Filippov solutions of discontinuous dynamical systems. There are various different notions for solutions of such systems. But the one introduced by Filippov is the one that better suits our discussion here. There are however more recent presentations of the main ideas for discontinuous dynamical systems that we will mainly follow in this section. In particular, we will take some basic material from \cite{cortes}. There are many applications in which discontinuous dynamical system occur. We refer to \cite{cortes} for a rather interesting discussion of relevant examples. Sliding modes is an important term in this field. 

The basic starting point to understand solutions to discontinuous dynamical systems (and of Filippov solutions in particular) is the notion of differential inclusion (\cite{smirnov})
\begin{equation}\label{incdif}
\mathbf{x}'(t)\in \F(\mathbf{x}(t)),\quad \mathbf{x}\in\R^N, \F:\R^N\mapsto 2^{\R^N},
\end{equation}
where now the right-hand side is not simply a map from $\R^N$ into $\R^N$, but rather a full subset of $\R^N$ is associated with each vector $\bx$ in $\R^N$. The differential inclusion itself asks for an absolutely continuous path $\mathbf{x}:J\to\R^N$ ($J$ is some fixed interval in $\R$) so that the derivative $\mathbf{x}'$ belongs to $\F(\mathbf{x})$ for a.e. $t\in J$. This kind of solutions are typically called Carath\'eodory solutions for (\ref{incdif}). As a matter of fact, one can consider Carath\'eodory solutions for discontinuous vector fields
\begin{equation}\label{disc}
\mathbf{x}'(t)=F(\mathbf{x}(t)),\quad t\in J
\end{equation}
by simply demanding that this equality holds for all times except a subset of vanishing measure, $\mathbf{x}$ is absolutely continuous, and $\mathbf{x}(0)=\mathbf{x}_0$. Alternatively, $\mathbf{x}$ is required to comply with the integral form of the system
$$
\mathbf{x}(t)=\mathbf{x}_0+\int_0^t F(\mathbf{x}(s))\,ds
$$
for all times $t\in J$. 

The existence of Carath\'eodory solutions for (\ref{incdif}) requires some properties of the set-valued map $\F$. A typical existence theorem is the following (\cite{smirnov}). We just state the autonomous version which suffices for our purposes in this work. 
\begin{theorem}\label{exis}
Suppose that the set-valued map $\F:\R^N\mapsto 2^{\R^N}$, taking on values in the Borel sets of $\R^N$, is upper-semicontinuous, and the sets $\F(\mathbf{x})$ are non-empty, compact and convex for all $\mathbf{x}\in\R^N$. Then  there are (local) solutions for (\ref{incdif}) starting from arbitrary points $\mathbf{x}_0\in\R^N$. 
\end{theorem}

Suppose we begin with a discontinuous dynamical system
\begin{equation}\label{disc}
\mathbf{x}'=F(\mathbf{x}),\quad F:\R^N\to\R^N,
\end{equation}
with $F$ not continuous, but has certain manifolds of discontinuity of various dimensions in such a way that surfaces of lower dimension are intersections of surfaces of higher dimension in a typical hierarchical fashion. Though more complicated situations can be considered, we will restrict attention to this kind of piecewise continuous vector fields. These actually cover many situations of practical interest. It is in fact the one treated explicitly in Theorem 4, pag. 115, of \cite{filippov}. 
One then defines the Filippov extension of $F$ to be the set-valued map $\F:\R^N\mapsto 2^{\R^N}$ given by the rule:
\begin{enumerate}
\item $\F(\mathbf{x})=\{F(\mathbf{x})\}$ if $F$ is continuous at $\mathbf{x}$;
\item $\F(\mathbf{x})$ is the convex hull of all possible accumulation points of $F(\mathbf{y})$ as $\mathbf{y}\to \mathbf{x}$.
\end{enumerate}
One main situation of interest takes place when $F$ is piecewise continuous. By this we mean that $F$ is continuous except for a closed set 
made up of a union of surface discontinuities. Some of these surfaces may intersect each other in lower dimensional manifolds. 
At each point of discontinuity, the possible accumulation points is a finite set of vectors so that the above existence Theorem \ref{exis} can be applied to obtain integral curves for the Filippov differential inclusion for $\F$. 

We are now concerned about uniqueness. For this we further assume that the surfaces of discontinuity of $F$ are smooth and non-singular, so that we can talk about their tangent planes at all of their points. In this situation we can be more precise about Filippov solutions of (\ref{disc}). We would like to make a particular selection $\tilde F$ of $\F$ for each point $\mathbf{x}$ of discontinuity of $F$, according to the following rule. 
\begin{quotation}
For each point $\mathbf{x}$ of discontinuity of $F$ that belongs to a certain smooth surface $S$ of discontinuity of the least possible dimension, we take $\tilde F(\mathbf{x})$ as the intersection of the Filippov extension $\F(\mathbf{x})$ with the tangent plane of $S$ at $\mathbf{x}$. 
\end{quotation}
Note that $\tilde F(\mathbf{x})$ could be empty or set-valued itself. Filippov solutions for (\ref{disc}) are then Carath\'eodory solutions for 
$$
\mathbf{x}'\in\tilde F(\mathbf{x}).
$$
In particular, the set of times for which $\tilde F(\mathbf{x})$ is empty has to be of null measure for solutions. 
Theorem 4, pag. 115 in \cite{filippov} is a uniqueness result for this kind of solutions. However, it holds under rather restrictive and technical hypotheses which are too rigid for our purposes. This comment is also valid for Proposition 5, pag. 53 of \cite{cortes}. 
Indeed, this uniqueness is somehow associated with stability in the sense that unstable equilibria admit several integral curves emanating from them. To avoid these difficulties, we  do not place ourselves in that context but, to make things simpler, admit that eventually we may have several integral curves through some starting points. 

Finally, (Filippov) equilibria for (\ref{disc}) are points $\mathbf{x}$ such that $\mathbf{0}\in\tilde F(\mathbf{x})$, and (locally) stable equilibria are such points $\mathbf{x}$, with $\mathbf{0}\in\tilde F(\mathbf{x})$, so that integral curves starting on neighboring points converge to $\mathbf{x}$, even though we may  have several integral curves starting from the same vector. 

\subsection{The case of a single inequality constraint}
We retake problem \eqref{probsen} and conditions \eqref{optsen}. Suppose both $f$ and $g$ are smooth, and $\nabla g$ never vanishes in a vicinity of the interface $g=0$ relative to $g\ge0$. When a certain field $\Psi$ is discontinuous through an interface $g=0$, we systematically assume that over $g=0$ is defined as its Filippov's extension. Recall that $\Psi\cdot\nabla g=0$ over $g=0$ by definition whenever the Filippov's extension has a non-empty intersection with the tangent plane. Otherwise, Fillipov's extension is empty. 

\begin{lemma}\label{segundo}
Suppose $\overline\bx$ is an isolated local solution of the optimization problem \eqref{probsen} and of \eqref{optsen} 
under the indicated hypotheses on $f$ and $g$. Let $\Psi:\R^N\to\R^N$ be a, possibly discontinuous, field through $g=0$, defined to be its Filippov's extension on $g=0$ such that:
\begin{enumerate}
\item $\Psi\cdot\nabla g<0$ in a neighbourhood of $g=0$ relative to $g\ge0$;
\item $\Psi\cdot\nabla f\le0$ if $g\le0$;
\item $\Psi$ is non-empty in a vicinity of $\overline\bx$, and $\Psi=\bcero$ and $g=0$ only happen when $\nabla f$ and $\nabla g$ are parallel.
\end{enumerate}
Then integral curves of the dynamical system $\bx'=\Psi(\bx)$ starting sufficiently close to $\overline\bx$ will converge to $\overline\bx$ as in Lemma \ref{primero}. 
\end{lemma}
\begin{proof}
If $g(\overline\bx)<0$, everything takes place in an open subset away from the manifold $g=0$, and the situation is similar to the classic one: $\nabla g$ does not play any role. 

If, on the contrary, $g(\overline\bx)=0$, the proof is similar to that of Lemma \ref{primero}. We use exactly the same idea. Seeking a contradiction, suppose there is some $r_0>0$, and a sequence $\cC_j$ of non-empty (pieces of) integral curves of the dynamical system associated with $\Psi$,  contained in the ball $\bbB(\overline\bx, r_0)$, and starting at $\bx_j$, such that 
$$
\bx_j\to\overline\bx,\quad \cC_j\cap\partial\bbB(\overline\bx, r_0)\neq\emptyset,\quad \cC_j\hbox{ moving away from }\overline\bx.
$$
Because of the  fact that $\Psi$ is tangent to the manifold $g=0$, $\cC_j$ cannot cross through $g=0$, and so the sign of $g$ can be assumed to be the same for all $\cC_j$. 
\begin{enumerate}
\item If such sign is positive $g(\cC_j)>0$, because of the first condition on $\Psi$ related to $\nabla g$, the limit set of $\cC_j$ must be contained in $g=0$.
\item If that sign is non-positive $g(\cC_j)\le0$, by the condition on $\Psi$ relative to $\nabla f$, and bearing in mind that do not have nearby a critical point of $f$ (case $g$ inactive already treated), the limit set of $\cC_j$ must be contained in $g=0$ as well.
\end{enumerate}
In any case, by continuity, as $j\to\infty$, there must be (a piece of) an integral curve contained in the manifold $g=0$ joining $\overline\bx$ with some point in $\partial\bbB(\overline\bx, r_0)$. But this is impossible, because $f$ is a Lyapunov function for $\Psi$ over $g=0$, and there is a unique equilibrium point of $\Psi$ over $g=0$ in the ball $\bbB(\overline\bx, r_0)$ if $r_0$ is sufficiently small, which coincides with $\overline\bx$ by the optimality conditions \eqref{optsen}.
\end{proof}

Our proposal in \eqref{propuesta} complies with all the requirements in the lemma. 
\begin{proposition}
Let $f$ and $g$ be smooth functions such that $\nabla g$ never vanishes in a neighbourhood of $g=0$. Then locally, asymptotically stable equilibria of the (Filippov's extension of the) discontinuous vector field $\Psi$ in \eqref{propuesta}  are exactly the local solutions of \eqref{probsen} or of \eqref{optsen}.
\end{proposition}
\begin{proof}
We first check that the discontinuous vector field $\Psi$ in \eqref{propuesta} verifies all of the requirements in the lemma in a neighbourhood of a local solution $\overline\bx$ of \eqref{optsen}. The first two are straightforward. Concerning the third, notice that at $\overline\bx$, as compared with feasible points $\bx$ with $g(\bx)\le0$, $g$ attains a maximum, while $f$ attains a minimum. In this way, $\nabla f\cdot\nabla g\le0$ in a neighbourhood of $\overline\bx$. This condition suffices to have that $\Psi$ can indeed be extended at points of discontinuity in such a way that it belongs to the tangent space. Clearly, such extension vanishes precisely when $\nabla f$ and $\nabla g$ are parallel. 
\end{proof}

\section{The case with a single inequality constraint and several equality restrictions}
Let us apply the previous ideas to the mathematical program
\begin{equation}\label{probgen00}
\hbox{Minimize in }\mathbf{x}\in\R^N:\quad f(\mathbf{x})\quad \hbox{subject to}\quad g(\mathbf{x})\le\bcero, \mathbf{h}(\mathbf{x})=\bcero,
\end{equation}
where $f:\R^N\to\R$, $g:\R^N\to\R$, $\mathbf{h}:\R^N\to\R^n$, are smooth.
Recall that we will be using the notation
$$
\mathbf{H}(\mathbf{x})\equiv\nabla \mathbf{h}(\mathbf{x})^T(\nabla \mathbf{h}(\mathbf{x})\nabla \mathbf{h}(\mathbf{x})^T)^{-1}\nabla \mathbf{h}(\mathbf{x})
$$
for the projection matrix onto the subspace generated by $\{\nabla \mathbf{h}(\mathbf{x})\}$. Let us designate by $\H(\mathbf{x})$ such subspace. 
For simplicity, some times we will drop the particular vector $\mathbf{x}$ where Jacobian matrices are supposed to be evaluated $\mathbf{H}=\mathbf{H}(\mathbf{x})$, $\H=\H(\mathbf{x})$, and so on. 
The matrix $\id-\mathbf{H}$ is the projection onto the orthogonal complement of $\H$. As such, note that 
$$
\mathbf{H}^2=\mathbf{H},\quad (\id-\mathbf{H})^2=\id-\mathbf{H},\quad (\id-\mathbf{H})\mathbf{H}=\mathbf{H}(\id-\mathbf{H})=\bcero.
$$
They both are symmetric matrices. 
The vector field we would like to consider is
$$
\Psi(\mathbf{x})=\begin{cases}-\mathbf{h}(\mathbf{x})\nabla \mathbf{h}(\mathbf{x})-\left(\id-\mathbf{H}(\mathbf{x})\right)\nabla g(\mathbf{x})^T,& g(\mathbf{x})>0,\\
-\mathbf{h}(\mathbf{x})\nabla \mathbf{h}(\mathbf{x})-\left(\id-\mathbf{H}(\mathbf{x})\right)\nabla f(\mathbf{x})^T,& g(\mathbf{x})\le0,
\end{cases}
$$
that underlies the mathematical program (\ref{probgen00}). For notational convenience, denote $\Psi^+(\mathbf{x})$ the limit value of $\Psi(\mathbf{x})$ as $\mathbf{x}$ tends to $g=0$ being positive, and similarly for $\Psi^-(\mathbf{x})$. 
Assuming that all functions involved here are continuously differentiable, and that the matrix $\nabla \mathbf{h}(\mathbf{x})\nabla \mathbf{h}(\mathbf{x})^T$ is always regular, it is clear, and has already been indicated, that the hypersurface of discontinuity for $\Psi$ is $g=0$. Even though $\Psi$ is defined precisely for $g=0$, these values are irrelevant and what matters is the limit value of $\Psi$ on the right, and on the left. We would like to prove a result similar to previous propositions.

\begin{theorem}\label{unades}
Let $f:\R^N\to\R$, $\mathbf{h}:\R^N\to\R^n$, $g:\R^N\to\R$ be continuously differentiable maps such that
\begin{itemize}
\item $\nabla \mathbf{h}(\mathbf{x})\nabla \mathbf{h}(\mathbf{x})^T$ is an invertible $n\times n$-matrix for all $\mathbf{x}$ in a neighbourhood of the manifold $\bh=\bcero$;
\item $\nabla g(\mathbf{x})$ does never belong to $\H(\mathbf{x})$ when $g(\mathbf{x})=0$.
\end{itemize}
Suppose, in addition, that the 
solutions of the Karush-Kuhn-Tucker optimality conditions
\begin{equation}\label{KKT}
\nabla f(\mathbf{x})+\mu\nabla g(\mathbf{x})+\lambda\nabla \mathbf{h}(\mathbf{x})=\bcero,\quad \mu\ge\bcero, \mu g(\mathbf{x})=0,
\end{equation} 
are isolated. 
Then, asymptotically stable equilibria for $\Psi$ (in the generalized sense of Filippov) are exactly the local solutions (minima) of the program (\ref{probgen00}). In particular, those equilibria $\mathbf{x}$ are
feasible points, $g(\mathbf{x})\le\bcero$, $\mathbf{h}(\mathbf{x})=\bcero$, which are solutions of (\ref{KKT}). 
\end{theorem}
\begin{proof}
Assume first that we have a certain admissible vector $\mathbf{x}\in\R^N$ which is a solution of (\ref{KKT}) with $\mu>0$. Then $g(\mathbf{x})=0$, $\mathbf{h}(\mathbf{x})=\bcero$, and $\nabla f(\mathbf{x})+\mu\nabla g(\mathbf{x})+\lambda\nabla \mathbf{h}(\mathbf{x})=\bcero$, for a certain $\lambda\in\R^m$. Because this point lies precisely at the surface of discontinuity for $\Psi$, its Filippov extension on the hypersurface $g=0$ is taken to be the intersection of the segment $[\Psi^+(\mathbf{x}), \Psi^-(\mathbf{x})]$ with the tangent plane of $g=0$ at that vector $\mathbf{x}$. The vectors $\Psi^{\pm}(\mathbf{x})$ are the two-sided limits of $\Psi$ at $\mathbf{x}$  for $g>0$, and $g<0$, respectively. Notice that
 $\mathbf{h}(\mathbf{x})=\bcero$ implies that
$$
\Psi^+(\mathbf{x})=-(\id-\mathbf{H}(\mathbf{x}))\nabla g(\mathbf{x})^T,\quad \Psi^-(\mathbf{x})=-(\id-\mathbf{H}(\mathbf{x}))\nabla f(\mathbf{x})^T,
$$
and, applying $(\id-\mathbf{H}(\mathbf{x}))$ to (\ref{KKT}), we obtain
\begin{equation}\label{convex}
\frac\mu{1+\mu}\Psi^+(\mathbf{x})+\frac1{1+\mu}\Psi^-(\mathbf{x})=\bcero.
\end{equation}
This identity implies that the null vector is a convex combination of $\Psi^+(\mathbf{x})$, and $\Psi^-(\mathbf{x})$.
Due to the fact that $\nabla g(\mathbf{x})$ does not belong to $\H(\mathbf{x})$, the vector $\Psi^+(\mathbf{x})$ cannot belong to the tangent space of $g=0$ at $\mathbf{x}$, and so the line through $\Psi^+(\mathbf{x})$ and $\Psi^-(\mathbf{x})$ can only intersect once that tangent plane. But (\ref{convex}) implies that such unique intersection has to be the vanishing vector, and so $\mathbf{x}$ is an equilibrium point for $\Psi$ (in the sense of Filippov). 

The situation with $g(\mathbf{x})<0$ corresponds to the case where the constraint $g\le 0$ is inactive, and it leads to the previous framework without inequality constraints. 

Conversely, suppose that $\mathbf{x}$ is  a vector for which the Filippov extension of $\Psi$ vanishes. We treat three different situations.
\begin{enumerate}
\item $g(\mathbf{x})>0$. In this situation, equilibria would not be stable. See below the discussion on stability. 
\item $g(\mathbf{x})<0$. This case, as indicated above, would lead to a solution of (\ref{KKT}) where the inequality is not active. We refer to the ideas discussed earlier for the situation with no inequality constraints.
\item $g(\mathbf{x})=0$. Such a vector $\mathbf{x}$ is a root of the Filippov extension to the surface of discontinuity $g=0$. As such, the null vector belongs to the convex hull of $\{\Psi^+(\mathbf{x}), \Psi^-(\mathbf{x})\}$, i. e.
$$
r\Psi^+(\mathbf{x})+(1-r)\Psi^-(\mathbf{x})=\bcero, \quad r\in[0, 1].
$$
It is interesting to point out that $r$ cannot be $1$ because $\nabla g(\mathbf{x})$ cannot belong to the subspace $\H(\mathbf{x})$. Hence we can write this identity in the form
\begin{equation}\label{combconv}
\frac r{1-r}\Psi^+(\mathbf{x})+\Psi^-(\mathbf{x})=\bcero, \quad r\in[0, 1).
\end{equation}
This identity leads immediately to (\ref{KKT}) with $\mu=r/(1-r)\ge0$. It only remains to argue that necessarily $\mathbf{h}(\mathbf{x})=0$. To this end, observe that once we know that $\nabla f(\mathbf{x})+\mu\nabla g(\mathbf{x})$ belongs to $\H(\mathbf{x})$, then $(\id-\mathbf{H}(\mathbf{x}))(\nabla f(\mathbf{x})^T+\mu\nabla g(\mathbf{x})^T)=\bcero$, and 
$$
0=\frac r{1-r}\Psi^+(\mathbf{x})+\Psi^-(\mathbf{x})=-\frac r{1-r} \mathbf{h}(\mathbf{x})\nabla \mathbf{h}(\mathbf{x})-\mathbf{h}(\mathbf{x})\nabla \mathbf{h}(\mathbf{x}),
$$
because the other contributions drop off. Since the gradients of the components of $h$ are linearly independent, we must have $\mathbf{h}(\mathbf{x})=\bcero$.  
\end{enumerate}

Let us now turn to the important stability issue. As before, $g(\mathbf{x})<0$ takes us back to the situation with no inequality constraint that has already been examined. 
Assume first that $g(\mathbf{x})>0$. 
We would have that
$$
\mathbf{h}(\mathbf{x})\nabla \mathbf{h}(\mathbf{x})+\left(\id-\mathbf{H}(\mathbf{x})\right)\nabla g(\mathbf{x})^T=\bcero.
$$
But in a certain neighborhood of such $\mathbf{x}$, the field $\Psi$ is given by 
$$
\Psi=-\mathbf{h}\nabla \mathbf{h}-\left(\id-\mathbf{H}\right)\nabla g^T, 
$$
so that $\langle\Psi, \nabla(|\mathbf{h}|^2)\rangle\le0$. In this way, integral curves for $\Psi$ will decrease $|\mathbf{h}|^2$, and hence if $\mathbf{h}(\mathbf{x})$ is not zero the equilibrium would not be stable. Over the manifold $\mathbf{h}=0$, $\Psi=-\left(\id-\mathbf{H}\right)\nabla g^T$, and this implies that $g$ should decrease along this manifold. Hence if $g(\mathbf{x})>0$, the equilibrium point $\mathbf{x}$ cannot be stable. 

Suppose instead that (\ref{combconv}) holds, and, in addition, $\mathbf{h}(\mathbf{x})=\bcero$. It is clear that
\begin{equation}\label{semilya}
\langle\Psi^+(\mathbf{x}), \nabla g(\mathbf{x})\rangle=-\left|(\id-\mathbf{H}(\mathbf{x}))\nabla g(\mathbf{x})^T\right|^2<0
\end{equation}
because $\mathbf{h}(\mathbf{x})=\bcero$, $\nabla g(\mathbf{x})$ does not belong to $\H(\mathbf{x})$, and due to the fact that 
$$
\langle(\id-\mathbf{H}(\mathbf{x}))\nabla g(\mathbf{x})^T, \mathbf{H}(\mathbf{x})\nabla g(\mathbf{x})^T\rangle=0.
$$
Condition (\ref{semilya}) implies that $g$ acts like a Lyapunov function for $\Psi$ when $g>0$ because integral curves for $\Psi^+$ will converge to the hypersurface $g=0$. Geometrically, (\ref{semilya}) means that $\Psi^+(\mathbf{x})$ points towards the region $g<0$. But then condition (\ref{combconv}) implies that $\Psi^-(\mathbf{x})$ points towards the region $g>0$, and so integral curves for $\Psi^-$ near $\mathbf{x}$ will converge to the hypersurface $g=0$ as well. On the other hand, it is also straightforward to check that
$$
\langle\Psi^{\pm}, \nabla |\mathbf{h}(\mathbf{y})|^2\rangle =-2|\mathbf{h}(\mathbf{y})|^2|\nabla \mathbf{h}(\mathbf{y})|^2
$$
for all $\mathbf{y}$, and so integral curves at both sides of the singular hypersurface converge to $\mathbf{h}=\bcero$ as well. We can, therefore, conclude that integral curves for $\Psi$ will accumulate on the surface $g=0$, $\mathbf{h}=\bcero$. Let us finally check that on this surface, when the Filippov extension of $\Psi$ holds so that the segment with end-points $\Psi^+$ and $\Psi^-$ intersects the tangent hyperplane to the hypersurface $g=0$, $f$ is a Lyapunov function. Indeed, suppose that ($\mathbf{h}=\bcero$)
$$
\langle r(\id-\mathbf{H})\nabla g^T+(1-r)(\id-\mathbf{H})\nabla f, \nabla g^T\rangle=0
$$
for some $r\in[0, 1)$. We find directly the expressions
$$
r=\frac{-\langle (\id-\mathbf{H})\nabla f^T, \nabla g\rangle}{\langle (\id-\mathbf{H})(\nabla g^T-\nabla f^T), \nabla g\rangle},\quad 
1-r=\frac{\langle(\id-\mathbf{H})\nabla g^T, \nabla g\rangle}{\langle (\id-\mathbf{H})(\nabla g^T-\nabla f^T), \nabla g\rangle}.
$$
The condition $1-r\in[0, 1]$ implies that the denominator in these two fractions ought to be strictly positive. But then, after some careful arithmetic, and bearing in mind that $\mathbf{H}^2=\mathbf{H}$, $(\id-\mathbf{H})^2=(\id-\mathbf{H})$, $\mathbf{H}^T=\mathbf{H}$, $(\id-\mathbf{H})^T=(\id-\mathbf{H})$, we find that $\langle r(\id-\mathbf{H})\nabla g^T+(1-r)(\id-\mathbf{H})\nabla f^T, \nabla f\rangle$ is given by
$$
\frac{\left(|(\id-\mathbf{H})\nabla f^T|^2|(\id-\mathbf{H})\nabla g^T|^2-
(\langle(\id-\mathbf{H})\nabla f^T, (\id-\mathbf{H})\nabla g^T\rangle)^2\right)}{\langle (\id-\mathbf{H})(\nabla g^T-\nabla f^T), \nabla g\rangle}\ge0.
$$
Hence we have that $\langle\Psi, \nabla f\rangle\le0$, as desired. 
\end{proof}

\section{The full general situation}
From now on, we focus on a general mathematical program involving both equality and inequality constraints
\begin{equation}\label{probgenn}
\hbox{Minimize in }\mathbf{x}\in\R^N:\quad f(\mathbf{x})\quad \hbox{subject to}\quad \mathbf{g}(\mathbf{x})\le\bcero, \mathbf{h}(\mathbf{x})=\bcero,
\end{equation}
where $f:\R^N\to\R$, $\mathbf{g}:\R^N\to\R^m$, $\mathbf{h}:\R^N\to\R^n$, are continuous differentiable. The situation is considerably more complex. Let us recall the well-known Karush-Kuhn-Tucker optimality conditions that local solutions of this optimization problem must verify (under suitable constraint qualifications), and that we would like to reproduce for equilibria of a suitable vector field. Those are
$$
\nabla f(\mathbf{x})+\mu\nabla \mathbf{g}(\mathbf{x})+\lambda\nabla \mathbf{h}(\mathbf{x})=\bcero,\quad \mu\ge\bcero, \mu \mathbf{g}(\mathbf{x})=0.
$$
Bear in mind that $\nabla f$ is a $1\times N$-matrix (a vector in $\R^N$), while $\nabla \mathbf{g}$ and $\nabla \mathbf{h}$ are $m\times N$-, and $n\times N$ matrices, respectively; $\mu\in\R^m$, $\lambda\in\R^n$. In addition, the point $\mathbf{x}$ must, of course, be feasible $\mathbf{g}(\mathbf{x})\le\bcero$, $\mathbf{h}(\mathbf{x})=\bcero$. 

As we have seen in the preceding section, it is not possible to define a corresponding continuous vector field as in the equality-constrained problem. However, it is possible to come up with a discontinuous, but piecewise continuous, vector field which has the required properties. Our proposal for the full general situation is the following. Put 
$$
G(\mathbf{x})=\max{g_i(\mathbf{x})},\quad \I(\mathbf{x})=\{i: g_i(\mathbf{x})=G(\mathbf{x})\},\quad \G(\mathbf{x})=\sum_{i\in\I(\mathbf{x})}\nabla g_i(\mathbf{x}).
$$
We will take  
$$
\Psi(\mathbf{x})=\begin{cases}-\mathbf{h}(\mathbf{x})\nabla \mathbf{h}(\mathbf{x})-\left(\id-\mathbf{H}(\mathbf{x})\right)\G(\mathbf{x})^T,& G(\mathbf{x})>0,\\
-\mathbf{h}(\mathbf{x})\nabla \mathbf{h}(\mathbf{x})-\left(\id-\mathbf{H}(\mathbf{x})\right)\nabla f(\mathbf{x})^T,& G(\mathbf{x})\le0.
\end{cases}
$$
If is clear that this $\Psi$ is not continuous, although it is piecewise-continuous. Among the manifolds of discontinuity we find the ones determined when some of the inequality constraints become active. 

We would like to prove a result similar to Theorem \ref{unades} but for the general situation when we have various inequality constraints. The discussion in the proof is similar to that of the previous theorem but more involved. In particular, we need to have a clear understanding of the Filippov extension of $\Psi$ to the surfaces of discontinuity. Note that in this situation we have surfaces of discontinuity of all dimensions $N-1$, $N-2$, \dots, $1$, $0$. 

There are two different kinds of points where the field $\Psi$ is discontinuous: either $G(\mathbf{x})>0$ but at least two of the components of $g(\mathbf{x})$ have the same value than $G(\mathbf{x})$; or else, $G(\mathbf{x})=0$. In the region $G(\mathbf{x})<0$, $\Psi$ is continuous. Let us try to understand the Filippov extension of $\Psi$ in those two situations.
\begin{enumerate}
\item $G(\mathbf{x})>0$, and the  cardinal of the set $\I(\mathbf{x})$ is at least two. It is elementary to check that 
the Filippov extension of $\Psi(\mathbf{x})$ is
$$
-\mathbf{h}(\mathbf{x})\nabla \mathbf{h}(\mathbf{x})-\hbox{co}(\{(\id-\mathbf{H}(\mathbf{x}))\nabla g_i(\mathbf{x})^T: i\in\I(\mathbf{x})\}),
$$
where co indicates the convex hull of the set. Suppose that the null vector belongs to this extension so that
$$
\mathbf{0}=-\mathbf{h}(\mathbf{x})\nabla \mathbf{h}(\mathbf{x})-\sum_{i\in\I(\mathbf{x})} \alpha_i(\id-\mathbf{H}(\mathbf{x}))\nabla g_i(\mathbf{x}).
$$
Because the first term belongs to $\H(\mathbf{x})$ but the second one lies in its orthogonal complement, both terms must vanish, and so $\mathbf{h}(\mathbf{x})=0$, and
$$
\sum_{i\in\I(\mathbf{x})} \alpha_i(\id-\mathbf{H}(\mathbf{x}))\nabla g_i(\mathbf{x})^T=\bcero.
$$
We would like these equilibria to be unstable. 
\item $G(\mathbf{x})=0$. In this case the Filippov extension is the convex hull of the set that incorporates, in addition to $\{(-\mathbf{h}(\mathbf{x})\nabla \mathbf{h}(\mathbf{x})-(\id-\mathbf{H}(\mathbf{x}))\nabla g_i(\mathbf{x})^T: i\in \I(\mathbf{x})\}$, the vector $-\mathbf{h}(\mathbf{x})\nabla \mathbf{h}(\mathbf{x})-(\id-\mathbf{H}(\mathbf{x}))\nabla f(\mathbf{x})^T$. We would like equilibria in this range of vectors to be stable. 
\end{enumerate}

We restate our most general result which is Theorem \ref{variasdes} in the Introduction.

\begin{theorem}\label{variasdess}
Let $f:\R^N\to\R$, $\mathbf{h}:\R^N\to\R^n$, $\mathbf{g}:\R^N\to\R^m$ be continuously differentiable maps such that
\begin{itemize}
\item $\nabla \mathbf{h}(\mathbf{x})\nabla \mathbf{h}(\mathbf{x})^T$ is an invertible $n\times n$-matrix for all $\mathbf{x}\in\R^N$ in a neighbourhood of the set $\bh=\bcero$;
\item for each $\mathbf{x}$ where $G(\mathbf{x})=0$, the intersection of the subspaces $\H(\mathbf{x})$ and the one spanned by $\{\nabla g_i(\mathbf{x}): i\in\I(\mathbf{x})\}$ is the trivial subspace;
\item the null vector is never in the convex hull of the set $\{\nabla g_i(\mathbf{x}): i\in\I(\mathbf{x})\}$ when $G(\mathbf{x})=0$;
\item solutions of the corresponding Karush-Kuhn-Tucker optimality conditions
\begin{equation}\label{KKTdd}
\nabla f(\mathbf{x})+\mu\nabla \mathbf{g}(\mathbf{x})+\lambda\nabla \mathbf{h}(\mathbf{x})=\bcero,\quad \mu\ge\bcero, \mu \mathbf{g}(\mathbf{x})=0,
\end{equation}
complying with $\mathbf{g}(\mathbf{x})\le\bcero$, $\mathbf{h}(\mathbf{x})=\bcero$, are isolated.
\end{itemize}
Then,  asymptotically stable equilibria for $\Psi$ (in the generalized sense of Filippov) are exactly the feasible points $\mathbf{x}$, $\mathbf{g}(\mathbf{x})\le\bcero$, $\mathbf{h}(\mathbf{x})=\bcero$, which are solutions of (\ref{KKTdd}), and correspond to a local optimal solution (minima) of  (\ref{probgenn}). 
\end{theorem}
\begin{proof}
By the discussion prior to the statement of this result, 
it is clear that we would like to have stable equilibria when $G(\mathbf{x})=0$, and unstable equilibria for $G(\mathbf{x})>0$. The situation $G(\mathbf{x})<0$ takes us back to the equality-constraint case. 

Just as in the proof of Theorem \ref{unades}, we proceed in three steps. 

1. Assume $\mathbf{x}$ is a feasible solution of (\ref{KKTd}), so that $G(\mathbf{x})\le0$, but in fact $G(\mathbf{x})=0$. Then the strictly positive multipliers $\mu_i$ should correspond to $i\in\I(\mathbf{x})$. By applying $(\id-\mathbf{H}(\mathbf{x}))$ to the main vector equation, we obtain
$$
(\id-\mathbf{H})\nabla f^T+\mu(\id-\mathbf{H})\nabla \mathbf{g}^T=\bcero.
$$
The condition on the signs of $\mu$ just mentioned, implies that the null vector is a convex combination of the vectors generating the Filippov extension of $\Psi$ at such $\mathbf{x}$. Therefore, this vector $\mathbf{x}$ is an equilibrium point for $\Psi$ (in the sense of Filippov). Recall that $\mathbf{h}(\mathbf{x})=\bcero$. 

2. Conversely, suppose now that $\mathbf{x}$ is an equlibrium for the Filippov extension of $\Psi$ with $G(\mathbf{x})=0$. The condition of being an equlibrium for $\Psi$ at a point of discontinuity, implies that 
\begin{gather}
\sum_{i\in\I(\mathbf{x})}\alpha_i\left(-\mathbf{h}(\mathbf{x})\nabla \mathbf{h}(\mathbf{x})-(\id-\mathbf{H}(\mathbf{x}))\nabla g_i(\mathbf{x})^T\right)+\alpha\left(-\mathbf{h}(\mathbf{x})\nabla \mathbf{h}(\mathbf{x})-(\id-\mathbf{H}(\mathbf{x}))\nabla f(\mathbf{x})^T\right)=\bcero,\nonumber\\
\alpha+\sum_{i\in\I(\mathbf{x})}\alpha_i=1, \alpha_i, \alpha\ge0.\nonumber
\end{gather}
Note that $\alpha>0$, for the situation $\alpha=0$ is impossible under our assumptions for the inequality constraints. 
This identity among the gradients of $f$ and the $g_i$'s leads to
$$
-\mathbf{h}(\mathbf{x})\nabla \mathbf{h}(\mathbf{x})=(\id-\mathbf{H}(\mathbf{x}))\left(\alpha\nabla f(\mathbf{x})^T+\sum_{i\in\I(\mathbf{x})}\alpha_i\nabla g_i(\mathbf{x})^T\right).
$$
But since the left-hand side belongs to $\H$, and the right-hand side to its orthogonal complement, conclude that in fact
$$
\bcero=\mathbf{h}(\mathbf{x})\nabla \mathbf{h}(\mathbf{x})=(\id-\mathbf{H}(\mathbf{x}))\left(\alpha\nabla f(\mathbf{x})^T+\sum_{i\in\I(\mathbf{x})}\alpha_i\nabla g_i(\mathbf{x})^T\right).
$$
In particular, because $\nabla \mathbf{h}(\mathbf{x})$ cannot be the null vector, $\mathbf{h}(\mathbf{x})=\bcero$. Besides
$$
\alpha\nabla f(\mathbf{x})+\sum_{i\in\I(\mathbf{x})}\alpha_i\nabla g_i(\mathbf{x})-\mathbf{H}(\mathbf{x})\left(\alpha\nabla f(\mathbf{x})^T+\sum_{i\in\I(\mathbf{x})}\alpha_i\nabla g_i(\mathbf{x})^T\right)=\bcero.
$$
Dividing through by $\alpha>0$, and renaming terms, we finally arrive at
$$
\nabla f(\mathbf{x})+\mu \nabla \mathbf{g}(\mathbf{x})+\lambda\nabla \mathbf{h}(\mathbf{x})=\bcero,
$$
for appropriate $\mu\ge\bcero$, and $\lambda$, with $\mu \mathbf{g}(\mathbf{x})=0$. 

3. We finally turn to the important stability issue. Let $\overline{\mathbf{x}}$ be such an equilibrium point for (the Filippov extension of) $\Psi$, and let $\mathbf{x}$ be a nearby point. We can have various possibilities. First, suppose that $\mathbf{h}(\mathbf{x})$ is not the zero vector. It is clear that, as indicated already earlier,  $\langle\Psi(\mathbf{x}), \nabla \mathbf{h}(\mathbf{x})\rangle$ is negative, so that integral curves will tend to a point where $\mathbf{h}$ vanishes. We can therefore put $\mathbf{h}(\mathbf{x})=\bcero$. Suppose, in addition, that $G(\mathbf{x})>0$ having $\mathbf{h}(\mathbf{x})=\bcero$ (in particular if $G(\overline{\mathbf{x}})>0$), because some $g_i(\mathbf{x})>0$, and $\I(x)=\{i\}$. Again, it is clear that $\langle\Psi(\mathbf{x}), \nabla g_i(\mathbf{x})\rangle\le0$ becasue $\Psi(\mathbf{x})=-(\id-\mathbf{H})\nabla g_i$. 
This could occur for more than one component when $G(\mathbf{x})=g_i(\mathbf{x})$ for several indices $i$. But the sets of those points $\mathbf{x}$'s in a vicinity of $\overline{\mathbf{x}}$ where this happens have zero measure, and by continuity we would also have that integral curves for $\Psi$ would have to decrease all those $g_i$'s. 
Therefore, integral curves would tend to $G=0$. In particular, stable equilibria $\overline{\mathbf{x}}$ cannot occur when $G$ is strictly positive as desired. 
If $G(\mathbf{x})<0$ (having $\mathbf{h}(\mathbf{x})=\bcero$), then $\langle\Psi(\mathbf{x}), \nabla f(\mathbf{x})\rangle\le0$ and integral curves move decreasing $f$. Finally, suppose $G(\mathbf{x})=0$, $\mathbf{h}(\mathbf{x})=\bcero$. In a suitable vicinity of $\overline{\mathbf{x}}$, $\I(\overline{\mathbf{x}})=\I(\mathbf{x})$ because these sets of indices are finite. Therefore, the Filippov extension in such points $\mathbf{x}$ will correspond to a certain convex combination of $\{\nabla g_i(\mathbf{x}): i\in \I(\overline{\mathbf{x}})\}$ and $\nabla f(\mathbf{x})$, belonging to the tangent space of the manifold $\{g_i(y)=0: i\in\I(\overline{\mathbf{x}})\}$ at $\mathbf{x}$. This implies that
$$
\langle(\id-\mathbf{H})\sum_i \mu_i\nabla g_i^T+(\id-\mathbf{H})\nabla f^T, (\id-\mathbf{H})\nabla g_i^T\rangle=\bcero,\quad i\in\I(\overline{\mathbf{x}}).
$$
We conclude in this case that
$$
\langle(\id-\mathbf{H})\sum_i \mu_i\nabla g_i^T+(\id-\mathbf{H})\nabla f^T, (\id-\mathbf{H})\nabla f^T\rangle\ge0,
$$
by the following elementary lemma.
\begin{lemma}
Let $\{\mathbf{v}_i\}$ be a finite collection of vectors in a certain finite-dimensional vector space $V$ endowed with a inner product $\langle\cdot, \cdot\rangle$. Suppose that for a certain convex combination $\mathbf{v}=\sum_i\alpha_i\mathbf{v}_i$ with $\alpha_1>0$, we have that
$$
\langle \mathbf{v}, \mathbf{v}_j\rangle=0, \hbox{ for all }j\ge2.
$$
Then
$$
\langle \mathbf{v}, \mathbf{v}_1\rangle=\frac{|\mathbf{v}|^2}{\alpha_1}\ge0.
$$
\end{lemma}
\end{proof}
\section{Linear programming}
Our last section is devoted to writing explicitly the field $\Psi$ for a linear programming problem, and check hypotheses so that Theorem (\ref{variasdes}) can be applied. As indicated in the Introduction, we plan to focus more fully in linear programming from this perspective in the near future.

Consider the general linear programming problem
\begin{equation}\label{proglin}
\hbox{Minimize in }\mathbf{x}\in\R^N:\quad \mathbf{c}\mathbf{x}\quad\hbox{subject to}\quad \mathbf{A}\mathbf{x}=\mathbf{a}, \mathbf{B}\mathbf{x}\le \mathbf{b},
\end{equation}
where $\mathbf{c}\in\R^N$, $\mathbf{A}\in\M^{n\times N}$, $\mathbf{a}\in\R^n$, $\mathbf{B}\in\M^{m\times N}$, $\mathbf{b}\in\R^m$. 
The field $\Psi$ reads in this particular situation
$$
\Psi(\mathbf{x})=\begin{cases}-(\mathbf{A}\mathbf{x}-\mathbf{a})\mathbf{A}-\left(\id-\mathbf{H}\right)\G(\mathbf{x})^T,& G(\mathbf{x})>0,\\
-(\mathbf{A}\mathbf{x}-\mathbf{a})\mathbf{A}-\left(\id-\mathbf{H}\right)\mathbf{c}^T,& G(\mathbf{x})\le0,
\end{cases}
$$
where  
\begin{enumerate}
\item $\mathbf{H}=\mathbf{A}^T(\mathbf{A}\mathbf{A}^T)^{-1}\mathbf{A}$ is a constant, $N\times N$-matrix;
\item $G(\mathbf{x})=\max\{\mathbf{B}\mathbf{x}-\mathbf{b}\}$, the maximum taken over the components of $\mathbf{B}\mathbf{x}-\mathbf{b}$;
\item $\G(\mathbf{x})=\sum_{i\in\I(\mathbf{x})}\mathbf{B}_i$ if $\mathbf{B}_i$ are the rows of $\mathbf{B}$, and $\I(\mathbf{x})=\{i: \mathbf{B}_i\mathbf{x}-b_i=G(\mathbf{x})\}$. 
\end{enumerate}
Then we have a corollary of Theorem \ref{probgenn} (also Theorem \ref{variasdes}) applied to (\ref{proglin}). 
\begin{corollary}
Suppose that
\begin{itemize}
\item the matrix $\mathbf{A}\mathbf{A}^T$ is an invertible $n\times n$-matrix;
\item for each $\mathbf{x}$ where $G(\mathbf{x})=0$, the intersection of the subspaces $\H$, spanned by the rows of $\mathbf{A}$, and the one spanned by the rows $i$ of $\mathbf{B}$ where $\mathbf{B}_i\mathbf{x}-b_i=G(\mathbf{x})$, is trivial;
\item the null vector is never in the convex hull of the set $\{\mathbf{B}_i\}$ for $i\in\I(\mathbf{x})$ when $G(\mathbf{x})=0$. 
\end{itemize}
Then every integral curve for $\Psi$ (in the Filippov sense) converges to the optimal solution of (\ref{proglin}), if such linear programming problem admits a unique solution. 
\end{corollary}
When (\ref{proglin}) has a unique solution, 
it will be a global equilibria for the corresponding discontinuous dynamical system.
If, on the contrary, problem (\ref{proglin}) does not admit a solution or it is unfeasible, then the field $\Psi$ cannot have equilibria, and integrals curves either have to veer off to infinity, or they may tend to limit cycles. It may be worth to explore the non-continuous dynamics underlying linear mathematical programming problems. Even for very simple examples, the dynamics may be far from trivial (see \cite{gughai}, \cite{llibreponce}). 


\end{document}